\documentclass[11pt]{amsart}

\usepackage[all]{xy}
\usepackage{graphics}
\usepackage{epsfig}
\usepackage{amsmath,amsthm}
\usepackage{amscd}
\usepackage{amsfonts,latexsym,amssymb}

\newcommand{\cE}{{\mathcal{E}}}   \newcommand{\cF}{{\mathcal{F}}}
\newcommand{\cG}{{\mathcal{G}}}   
   
   \newcommand{\cL}{{\mathcal{L}}}
\newcommand{\cM}{{\mathcal{M}}}   
\newcommand{\cO}{{\mathcal{O}}}

   \newcommand{\cV}{{\mathcal{V}}}

\newcommand{\hcV}{\hat{\mathcal{V}}}
\newcommand{\hcF}{\hat{\mathcal{F}}}
\newcommand{\hcG}{\hat{\mathcal{G}}}

\newcommand{\bz}{\bar z}
\newcommand{\bw}{\bar w}
\newcommand{\ba}{\bar a}

\newtheorem{proposition}{Proposition}[section]
\newtheorem{theorem}[proposition]{Theorem}
\newtheorem{lemma}[proposition]{Lemma}

\theoremstyle{definition}
\newtheorem{definition}[proposition]{Definition}
\newtheorem{remark}[proposition]{Remark}

\begin{document}

\title[Instantons and the Hodge Conjecture]
{Spin(7) instantons and the Hodge Conjecture for certain abelian
four-folds: a modest proposal}

\thanks{It is a pleasure to thank Bobby Acharya, who told me of Spin(7) instantons and gave me a
copy of Christopher Lewis' thesis. I am also grateful to Dominic
Joyce, Nigel Hitchin, M.S. Narasimhan, and Gang Tian for very
helpful comments.}

\dedicatory{Dedicated to S. Ramanan on the occasion of his seventieth birthday.}

\author[T. R. Ramadas]{T. R. Ramadas}

\address{Abdus Salam I.C.T.P.,
        11 Strada Costiera, Trieste 34014 Italy  }

\email{ramadas@ictp.it}

\date\today

\begin{abstract} The Hodge Conjecture is equivalent to a statement
about conditions under which a complex vector bundle on a smooth
complex projective variety admits a holomorphic structure. In the
case of abelian four-folds, recent work in gauge theory suggests an
approach using Spin(7) instantons. I advertise a class of examples
due to Mumford where this approach could be tested. I construct
explicit smooth vector bundles - which can in fact be constructed in
terms of of smooth line bundles - whose Chern characters are given
Hodge classes. An instanton connection on these vector bundles would
endow them with a holomorphic structure and thus prove that these
classes are algebraic. I use complex multiplication to exhibit
Cayley cycles representing the given Hodge classes. I find alternate
complex structures with respect to which the given bundles are
holomorphic, and close with a suggestion (due to G. Tian) as to how
this may possibly be put to use.
\end{abstract}

\maketitle

\section{Introduction}

Let $X$ be a smooth complex projective variety of dimension $n$, and
$c$ a rational $(p,p)$ cohomology class ($0<p<n$). The Hodge
Conjecture is that

\begin{description}

\item[H] there exist finitely many (reduced, irreducible) $(n-p)$-dimensional
subvarieties $Y_i$ and rational numbers $a_i$ such that $c=\sum_i
a_i [Y_i]$, where $[Y_i]$ is the (rational) cohomology class dual to
$Y_i$ . That is, $c$ is dual to a rational algebraic cycle.
\end{description}

This is equivalent to

\begin{description}
\item[V] there exists a holomorphic vector bundle $E$ such that its Chern character
$ch(E)$ is equal to a rational multiple of $c$ modulo (classes of)
rational algebraic cycles.
\end{description}

The second statement implies the first because the Chern character
of a holomorphic (and therefore algebraic) bundle factors through
the Chow ring of algebraic varieties. The converse also holds. In fact, as
Narasimhan pointed out to me,  it is known (\cite{M}) that the rational Chow ring is generated by
stable vector bundles.

Let $X$, $c$ be as above. By a theorem of Atiyah-Hirzebruch
(\cite{A-H}, page 19), the Chern character map $ch:K^0(X) \otimes
\mathbb{Q} \to H^{even} (X,\mathbb{Q})$ is a bijection, where
$K^0(X)$ is the Grothendieck group of (topological/smooth) vector
bundles on $X$. Thus we are assured of the existence of a smooth
bundle $E$ and and an integer $n > 0$ such that $ch(E )=rank
(E)+nc$.  A possible strategy to show that a given class $c$ is
algebraic suggests itself -- find a suitable such bundle $E$ and
then exhibit a holomorphic structure on it. This note is  written to
argue that recent progress in mathematical gauge theory, and in
particular the work of G. Tian and C. Lewis, makes this worth
pursuing, at least in the case of certain abelian four-folds. Such
an approach to the Hodge Conjecture for the case of Calabi-Yau
four-folds is surely known to the experts (and this has been
confirmed to me), but I have only been able to locate some coy
references. Claire Voisin (\cite{V}), following a similar approach,
has much more definitive \emph{negative} results in the case of
\emph{non-algebraic} tori.

Before proceeding, let us note that the known ``easy" cases of the
Hodge conjecture are proved essentially by the above method. First,
given an integral  class $c  \in H^2(X,  \mathbf{Z})$, a smooth
hermitian line bundle $L$ exists with (first) Chern class equal to
$c$. Given any real 2-form $\Omega$ representing $c$ there exists an
unitary connection on $L$ with curvature $-2\pi i \Omega$. If $c$ is
a $(1,1)$ class, it can be represented by an $\Omega$ which is
$(1,1)$. The corresponding connection defines a holomorphic
structure on $L$.  If $c$ is an integral $(n-1,n-1)$ class, the
strong Lefschetz theorem exhibits the dual class as a rational
linear combination of complete intersections.

What follows is the result of much trial and error and
computations - which I either only sketch or omit altogether -
using  {\it Mathematica}; the notebooks are available on request. (I used an exterior algebra package of Sotirios
Bonanos, available from
\verb"http://www.inp.demokritos.gr/~sbonano/." )

\section{Mumford's examples}

We consider Hodge classes on certain abelian four-folds. These examples are
due to Mumford (\cite{P}).

It is best to start with some preliminary algebraic number theory.
If $F$ is an algebraic number field, with $degree\ F =d$, the ring
of algebraic integers $\Lambda \equiv {\mathfrak o}_F$ is a free
$\mathbb{Z}$-module of rank $d$ which generates $F$ as a
$\mathbb{Q}$-vector space. If $V$ denotes the real vector space
$\mathbb{R} \otimes_{\mathbb{Q}} F$, then $\Lambda \subset V$ is a
lattice and $X_r=V/\Lambda$ is a real $d$-torus.

Let $L$ denote the Galois saturation of $F$ in $\bar{\mathbb{Q}}
\subset \mathbb{C}$. (That is, $L$ is the smallest subfield Galois
over $\mathbb{Q}$ and containing any (and therefore all) embeddings
of $F$.) Then $G=Gal(L/\mathbb{Q})$ acts transitively on the set $E$
of embeddings $\iota:L \hookrightarrow \mathbb{Q}$ by $(g,\iota)
\mapsto g(\iota) = g \circ \iota$ ($g \in G,\ \iota \in E$), and the
image by $\iota$ is the fixed field of the stabiliser of $\iota$.
Further, the map
$$
\bar{\mathbb{Q}} \otimes_{\mathbb{Q}} F \to \bar{\mathbb{Q}}^E
$$
given by $1 \otimes x \mapsto (\iota(x))_E$ is an isomorphism of
$\bar{\mathbb{Q}}$ vector spaces.

Turning to the real torus $X_r$:
\begin{enumerate}
  \item we have natural isomorphisms $H_1(X_r,\mathbb{Z})=\Lambda$
and $H_1(X_r,\bar{\mathbb{Q}})=\bar{\mathbb{Q}}^E$;
  \item $H^1(X_r,\bar{\mathbb{Q}})$ has basis $\{dt_\iota\}_E$, where $dt_\iota$
  is induced by the projection to the $\iota^{th}$ factor from
  $\bar{\mathbb{Q}}^E$.
\end{enumerate}

In what follows we will identify the real or complex cohomolgy of
$X_r$ with the corresponding spaces of translation-invariant forms
on $X_r$.

We will need the following result, whose proof is straightforward.
\begin{proposition} A one-form $\omega=\sum_\iota \omega_\iota dt_\iota$ represents
a rational class iff the coefficients $\omega_\iota$ belong to $L$
and satisfy the equivariance
$$
\omega_{g(\iota)}=g(\omega_\iota),\ g \in G
$$
Similarly, a two-form $\phi=\sum_{\iota,\kappa} \phi_{\iota,\kappa}
dt_{\iota} \wedge dt_{\kappa}$ (with the coefficients antisymmetric
functions of the two indices) represents a rational class iff
$$
\phi_{g(\iota),g(\kappa)}=g(\phi_{\iota,\kappa}),\ g \in G
$$
\end{proposition}

Suppose now that the embeddings $E$ occur in complex conjugate pairs
- $E=E' \sqcup E''$, with each $\iota \in E'$ corresponding to
$\bar{\iota} \in E''$. Then the map
$$
V=\mathbb{R} \otimes_{\mathbb{Q}} F (\hookrightarrow \mathbb{C}
\otimes_{\mathbb{Q}} F \sim \mathbb{C}^E) \to \mathbb{C}^{E'}
$$
is an isomorphism of real vector spaces and induces a
(translation-invariant) complex structure on $X_r$, which becomes a
complex torus, which we will denote simply $X$.

We turn now to specifics. Let $P=ax^4+bx^2+cx+d$ be an irreducible
polynomial with rational coefficients and all roots
$x_1,x_2,x_3,x_4$ real. We will suppose that the roots are numbered
such that $x_1>x_2>x_3>x_4$. Let $L_1/\mathbb{Q}$ be the splitting
field $L_1=\mathbb{Q}[x_1,x_2,x_3,x_4] \subset \mathbb{R}$. We
suppose that $P$ is chosen such that the Galois group is $S_4$. This
is equivalent to demanding that $[L_1:\mathbb{Q}]=24$. We set
$L\equiv L_1[i]$. This is a Galois extension of $\mathbb{Q}$, with
Galois group $S_4 \times \{e,\rho\}$, where $\rho$ is complex
conjugation.

Consider a cube, with vertices labeled as in the figure:
 \begin{figure}[here]
    \centering
    \includegraphics[width=2in]{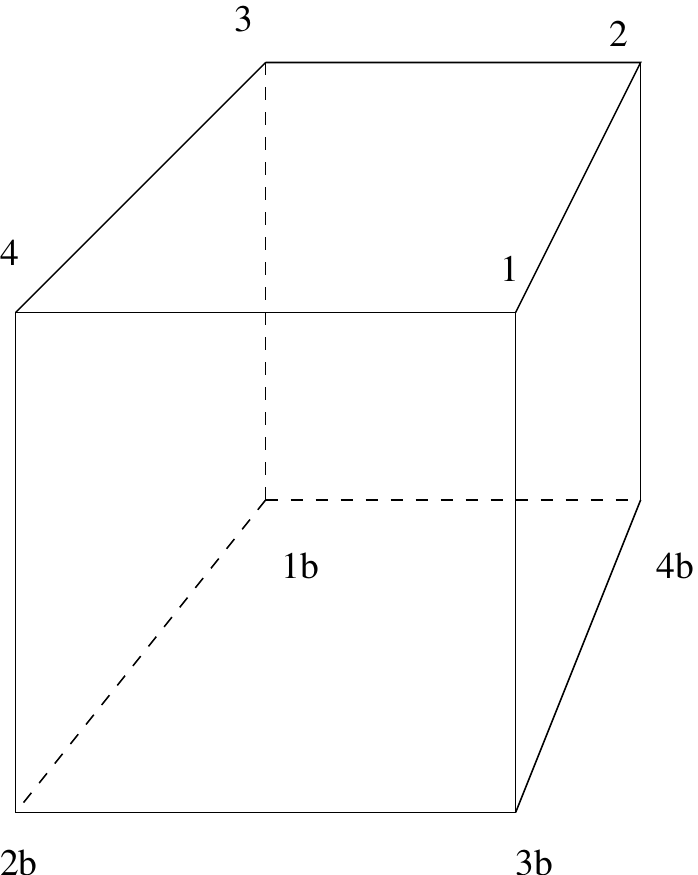}
    \label{cube}
 \end{figure}

Let $G$ denote the group of symmetries of the cube. We have the
exact sequence:
$$
1 \to \{e,\rho\} \to G \to S_4 \to 1
$$
where now $\rho$ denotes inversion, and $S_4$ is the group of
permutations of the four diagonals. Splitting this, identifying
$S_4$ with (special orthogonal) rotations implementing the
corresponding permutation of diagonals. we get an identification
$$
G \sim S_4 \times \{e,\rho\} = Gal(L/\mathbb{Q})
$$
Let $H$ denote the stabiliser of the vertex 1, $F$ the corresponding
fixed field, and $\varphi_1:F \to L \to   \mathbb{C}$ the
corresponding embedding.  The left cosets of $H$ can be identified
with the vertices of the cube, as well as embeddings of $F$ in
$\mathbb{C}$. We label the latter  $\varphi_j, \varphi_{\bar{j}}\ \
(j=1,2,3,4)$.

Note that the field $F$ is invariant under complex conjugation, which
therefore acts on it with fixed field $F_1$.  Clearly, $F_1=\mathbb{Q}[x_1]$ .
We set
$$
\mathbb{D}=(x_1-x_2)(x_1-x_3)(x_1-x_4)(x_2-x_3)(x_2-x_4)(x_3-x_4)
$$
Given our ordering of the roots, $\mathbb{D}>0$. Note that
$i\mathbb{D}\in F$, $F=F_1[i\mathbb{D}]$, and $\Delta \equiv
\mathbb{D}^2$ is a rational  number. We will assume that (after
multiplying all the $x_i$ by a common natural number if necessary)
$\Delta$ is an integer (and so $\mathbb{D}$ is an algebraic integer).
We will repeatedly use the fact that the Galois conjugates of  $i\mathbb{D}\in F$ are given by
\begin{equation}
\begin{split}
\phi_j(i\mathbb{D})&=-(-1)^j i\mathbb{D}\\
\phi_{\bar{j}}(i\mathbb{D})&=(-1)^j i\mathbb{D}\\
\end{split}
\end{equation}

In our case $X_r$ is a real 8-torus. The embeddings $\varphi_i:F \to
\mathbb{C}$ induce $\mathbb{R}$-linear maps $z_i:V \to \mathbb{C}$,
such that $\mathbf{z}=(z_1,z_2,z_3,z_4)$ is an isomorphism of
$\mathbb{R}$-vector spaces $V \to \mathbb{C}^4$. We let $X$ denote
the complex manifold $V/\Lambda$ obtained thus. Note that if $a \in
{\mathfrak o}_F$, multiplication by $a$ is a $\mathbb{Q}$-linear map
$F \to F$ which induces a $\mathbb{R}$-linear map $V \to V$ taking
the lattice $\Lambda$ to itself. If
$\mathbf{z}(a)=(a_1,a_2,a_3,a_4)$, and $u\in V$ with
$\mathbf{z}(u)=(z_1,z_2,z_3,z_4)$ we also have
$\mathbf{z}(au)=(a_1z_1,a_2z_2,a_3z_3,a_4z_4)$, so that we see that
this induces an analytic map (in fact an isogeny) $X \to X$.  In
other words, ${\mathfrak o}_F$ acts on $X$ by ``complex
multiplication".

As a complex torus, $X$ is certainly K\"ahler, and we shall see
below that it is algebraic. What is relevant for our purposes is
that it is possible to describe explicitly the Hodge decomposition
as well as the rational structure of the complex cohomology of $X$.
Let $T$ (for ``top") denote the set of indices $\{1,2,3,4\}$ and $B$
(for ``bottom") the indices $\{\bar{1},\bar{2},\bar{3},\bar{4}\}$.
(The corresponding vertices are denoted $1b$, etc. in the figure.)

\begin{proposition} A basis of $H^{p,q}$ is labeled by subsets $P \subset T$,
$Q \subset B$, with $|P|=p$, and $|Q|=q$, and given by the translation-invariant
forms $dz^Pd\bz^Q$, where for example, if $P=\{i,j\}$, with $i<j$ we
set $dz^P=dz_i \wedge dz_j$, and if $Q=\{\bar{i},\bar{j}\}$ (again
with $i<j$), we set $d\bz^Q= d\bz_i \wedge d\bz_j$. A basis of the
rational cohomology $H^r_{\mathbb{Q}}$ is labelled by pairs $(R,
\chi)$ where
\begin{itemize}

\item $R$ is an orbit of $G$ in the set of sequences $\mu \equiv (\mu_1,\dots,\mu_r)$
of distinct elements in $T \cup B$, and

\item $\chi$ runs over a $\mathbb{Q}$-basis of $H_R$, the space of  $G$-equivariant
maps $R \to L$, satisfying
$$
\chi(\mu_{\sigma(1)},\dots,\mu_{\sigma(r)})=sign(\sigma)
\chi(\mu_1,\dots,\mu_r)
$$
for any permutation $\sigma$ such that $\mu, \mu_{\sigma} \in R$.
\end{itemize}
The corresponding classes are given by the forms
$$
\sum_{\mu \in R} \chi(\mu) dz^\mu
$$
We use the notation $dz^\mu=dz_{\mu_1}\wedge\dots\wedge dz_{\mu_r}$,
with the convention that $dz_{\bar{1}}=d\bz_1, etc.$
\end{proposition}

It is useful to note the following

\begin{lemma}
Given $R$, the $\mathbb{Q}$-dimension of $H_R$ is $|R|/r!$.
\end{lemma}

Note that if $r=2p$, a rational class as above is of type $(p,p)$
iff the orbit consists of sequences with elements equally divided
between the top and bottom faces of the cube. In particular, the
rational $(1,1)$ classes correspond to the $G$-orbit of the sequence
$(1,\bar{1})$. Since in this case $H_{R}$ has dimension 4, we see
that the Neron-Severi group has rank 4.

Consider now the orbit of the sequence $(1,3,\bar{2},\bar{4})$. This
corresponds to a two-dimensional space $\cM$ of rational $(2,2)$
classes, which have the property that \emph{these are not products
of rational $(1,1)$ classes.} It is easy to check that but for (the
$\mathbb{Q}$-span of) these, rational $(2,2)$ classes are generated
by products of rational $(1,1)$ classes.

\begin{proposition}
A $\mathbb{Q}$-basis of $\cM$ is given by the classes

\begin{itemize}
\item $M= \mathbb{D}(dz_1 d\bz_2 dz_3 d\bz_4 + d\bz_1 dz_2 d\bz_3 dz_4)$
\item $M'= i(dz_1 d\bz_2 dz_3 d\bz_4 - d\bz_1 dz_2 d\bz_3 dz_4)$
\end{itemize}

\end{proposition}

\noindent So the Hodge conjecture in this case would be that : {\bf
the classes $M$ and $M'$ are algebraic.}

We will use complex multiplication in an essential way later; here I illustrate its
use by showing how it can be used to halve our work. Consider
multiplication by the algebraic integer $a=1+i\mathbb{D} \in
{\mathfrak o}_F$. This induces a (covering) map $\pi_a:X \to X$ and
one easily computes:
\begin{equation}
\label{isogeny}
\begin{split}
\pi_a^* M &= ((1-\Delta)^2-4\Delta)M+4(1-\Delta)\Delta M' \\
\pi_a^* M' &= ((1-\Delta)^2-4\Delta)M'-4(1-\Delta)M\\
\end{split}
\end{equation}
This proves

\begin{proposition}
Algebraicity of either one of $M$ or $M'$ implies that of the other.
\end{proposition}

Before moving on, we find a positive rational $(1,1)$ form $\omega$
on $X$, which will show that it is projective. Let $\mu_1\in F_1$
(to be chosen in a moment) and consider the form
$$
\omega=\frac{i\mathbb{D}}{\Delta}(\mu_1 dz_1 d\bz_1-\mu_2 dz_2
d\bz_2+\mu_3 dz_3 d\bz_3-\mu_4 dz_4 d\bz_4)
$$
where $\mu_i$ are Galois conjugates. Clearly this is a rational
$(1,1)$ form, and it will be positive provided $(-1)^{j+1} \mu_j>0$.
For example, we can take $\mu_1=(x_1-x_2)(x_1-x_3)(x_1-x_4)$, and we
will do so. With this choice the holomorphic four-form $\theta
\equiv (1/\mathbb{D})dz_1 dz_2 dz_3 dz_4$ satisfies
\begin{equation}
\label{omegatheta} \frac{\omega^4}{4!}=\theta \wedge \bar{\theta}
\end{equation}

\section{Expressing $M,\ M'$ in terms of Chern characters}

Consider the $G$-orbit of $(1,3)$. The corresponding subspace of
$H^2_{\mathbb{Q}}$ is spanned by the classes of the form
$$
A_1=a_{13}(x_1-x_3)dz_1dz_3+....
$$
where $a_{13}$ belongs to the fixed field of the subgroup of $G$
that leaves the set of vertices $\{1,3\}$ invariant, and this
coefficient determines the others in the sum by Galois covariance.
We introduce the notation
$$
T_a=a_{13}a_{\bar{2}\bar{4}}(x_1-x_3)(x_2-x_4)-a_{1\bar{2}}a_{3\bar{4}}(x_1-x_2)(x_3-x_4)
+a_{1\bar{4}}a_{3\bar{2}}(x_1-x_4)(x_3-x_2)
$$
Squaring $A_1$, we get
\begin{equation*}
\begin{split}
A_1^2=&2a_{13}a_{24}(x_1-x_3)(x_2-x_4) dz_1dz_3dz_2dz_4+..\\
+&2a_{\bar{1}2}a_{13}(x_1-x_2)(x_1-x_3)d\bz_1dz_2dz_1dz_3+...\\
+&2a_{1\bar{2}}a_{2\bar{1}}(x_1-x_2)(x_2-x_1)dz_1d\bz_2dz_2d\bz_1+..\\
+&2T_a dz_1dz_3d\bz_2d\bz_4+..
\end{split}
\end{equation*}
If we make the replacement $a_{13} \rightsquigarrow
ic\mathbb{D}a_{13}$ ($c$ an integer introduced for later use in \S
7), we get a class $A_2$, such that
 \begin{equation*}
\begin{split}
A_2^2/(c^2\Delta)=&2a_{13}a_{24}(x_1-x_3)(x_2-x_4) dz_1dz_3dz_2dz_4+..\\
+&2a_{\bar{1}2}a_{13}(x_1-x_2)(x_1-x_3)d\bz_1dz_2dz_1dz_3+...\\
+&2a_{1\bar{2}}a_{2\bar{1}}(x_1-x_2)(x_2-x_1)dz_1d\bz_2dz_2d\bz_1\\
-&2T_a dz_1dz_3d\bz_2d\bz_4-..
\end{split}
\end{equation*}

Suppose now that the classes $A_i$ are integral.   (This is easily
arranged by clearing denominators.) Let $L_i$ ($i=1,2$) be the line
bundle with Chern class $A_i$.

\begin{proposition} Let $\cV_i=L_i \oplus L_i^{-1} ,\ i=1,2$. Then
\begin{equation*}
ch(\cV_1^{c^2\Delta} \ominus \cV_2)=4c^2\Delta(T_a dz_1dz_3d\bz_2d\bz_4+..)\\
\end{equation*}
 where the equality is modulo (rational) 0- and 8-forms.
\end{proposition}

We have the freedom to choose the coefficient $a_{13}$, which by
Galois covariance determines the other coefficients, and hence the
above classes. We now make the choice
$$
a_{13}=h_3
$$
where for later use we introduce the notation
\begin{equation}
\begin{split}
h_2=(x_1 x_2+x_3 x_4)\\
h_3=(x_1 x_3+x_2 x_4)\\
h_4=(x_1 x_4+x_2 x_3)
\end{split}
\end{equation}
Then $T_a=-\mathbb{D}$, and we get

\begin{theorem}\label{key}  With the above choice,
\begin{equation*}
ch(\cV_1^{c^2\Delta} \ominus \cV_2)=4c^2\Delta M
\end{equation*}
where the equality is modulo (rational) 0- and 8-forms.
\end{theorem}

The virtual bundle $\cV_1^{c^2\Delta} \ominus \cV_2$ has the properties:
$c_1=0$, and $c_2 \wedge \omega^2=0$, where $\omega$ is the rational
K\"ahler class defined at the end of \S 2. (This is because the
$M_i$, as can be easily seen, are orthogonal to $\omega$.) This will
not do for reasons to do with the Bogomolov inequality, but this can
be fixed because of a minor miracle:

\begin{proposition}\label{miracle1}  With the above choices,
\begin{equation*}
\begin{split}
A_1^2 \wedge \omega = -2i\Delta \frac{1}{\mu_4}
dz_1d\bz_1dz_2d\bz_2dz_3d\bz_3+...
 \end{split}
\end{equation*}
In particular, $A_1^2 \wedge \omega $ is a (rational) (3,3) form.
\end{proposition}

For later use, we also record
\begin{proposition}\label{miracle2}  With the above choices,
\begin{equation*}
\begin{split}
A_1 \wedge \omega^3 = 0\\
A_2 \wedge \omega^3 = 0\\
\end{split}
\end{equation*}
\end{proposition}

We will suppose that $k\omega$ (for some positive integer $k$) is an
integral class, and let $L_{k\omega}$ denote a (holomorphic, in fact
ample) line bundle with this Chern class. The following is a easy
consequence of \ref{miracle1}.

\begin{theorem}\label{bundles} Let $\hcV_1=L_1 \otimes L_{k\omega} \oplus L_1^{-1} \otimes L_{k\omega} $,
and $\hcV_2= L_2 \oplus L_2^{-1}$ and set $\cE=\hcV_1^{c^2\Delta}
\ominus \hcV_2$. Then
$$
ch(\cE)= 2c^2\Delta k \omega+4c^2\Delta M+k^2c^2\Delta \omega^2
$$
where the equality is modulo (rational) 0-, (3,3)- and 8-forms.
\end{theorem}

In particular, this (difference) bundle $\cE$ satisfies the
``Bogomolov inequality":
\begin{equation*}
\begin{split}
<c_2\omega^2>-\frac{2\Delta-3}{4(\Delta-1)}<c_1^2 \omega^2>&=\frac{c^2\Delta}{c^2\Delta-1}k^2<\omega^4>\\
&> 0
\end{split}
\end{equation*}

The symbol $<..>$ stands for integration against the fundamental
class. We use the quote marks since we are not (yet!) talking of a
holomorphic bundle $\cE$. Since the virtual bundle has positive
rank, we are justified, up to some non-canonical choices, in
dropping the qualifiers ``virtual''/``difference''.

\begin{remark} We have concentrated on the Hodge class $M$ in this section; it is possible, with slight modifications to the above expressions, to find a smooth bundle $\cE'$ whose Chern character similarly contains the Hodge class $M'$.
\end{remark}

\section{$Spin(7)$ instantons}

In this section we recall the definition of $Spin(7)$ instantons
(\cite{B-K-S}, \cite{T}), specialised to the case of a K\"ahler
four-fold $X$ with trivial canonical bundle $K_X$. We fix a
Ricci-flat K\"ahler form $\omega$, and let $\theta$ denote a
trivialisation of $K_X$ satisfying (\ref{omegatheta}). We define a
(complex antilinear) endomorphism $\star:\Omega^{(0,2)} \to
\Omega^{(0,2)}$, by
$$
|\alpha|^2 \theta=\alpha \wedge \star \alpha
$$
We have $\star^2=1$, so we can decompose the bundle into a self-dual
and anti-self-dual part:
$$
\Omega^{(0,2)}=\Omega^{(0,2)}_+ \oplus \Omega^{(0,2)}_-
$$

Let $E$ be a hermitian ($C^{\infty}$) vector bundle on $X$. A
Spin(7) instanton is a hermitian connection  $A$ on $E$, whose
curvature $F$ satisfies
$$
F_+^{(0,2)}=0,\  \  \Lambda F=0
$$
Here $\Lambda$ denotes as usual contraction with the K\"ahler form.
A crucial point is the following (\cite{T},\cite{L}):
\begin{proposition}
The $L^2$-norm of the curvature of a Spin(7) instanton satisfies
$||F_-^{(0,2)}||_2^2=\int Tr (F \wedge F) \wedge \bar{\theta}$
\end{proposition}

In particular, if the invariant on the right vanishes, a Spin(7)
instanton is equivalent to a holomorphic structure on $E$  together
with a Hermite-Einstein connection. Clearly, such a bundle would be
poly-stable, and hence (or directly from  the Hermite-Einstein
condition) satisfy the Bogomolov inequality:
\begin{equation}
\label{Bogomolov} c_2(E).\omega^2 \ge  \frac{r - 1}{2r}
c_1(E)^2.\omega^2
\end{equation}
where $r$ denotes the rank of $E$.

\emph{Now that we have embedded the problem of construction a
holomorphic structure on $\cE$ in a broader context -- that of
constructing an instanton connection -- one can envisage deforming
the complex structure in such a way
$$
\int c_2(\cE) \wedge \bar{\theta} \ne 0
$$
and still hope to have the moduli space of semi-stable holomorphic
structures on $\cE$ deform as the moduli space of instanton
connections.}

There are several possible approaches to the construction of such a
connection.

\begin{enumerate}
\item Exhibit an instanton by glueing.
\item The fact that the bundles are exhibited as a difference of two vector bundles,
each of which is in turn a sum of explicit line bundles, suggests the use of monads,
possibly combined with a twistor construction. This would involve a matrix of sections
of line bundles.
\end{enumerate}

A third idea, suggested to me by G. Tian, is pursued in the last
section of this paper.

\section{Calibrations; Cayley submanifolds}

In his thesis, C. Lewis \cite{L} shows how (in one particular case)
one can construct an instanton by glueing around a suitable Cayley
submanifold. (See also \cite{B}.) We define these terms below, and
then exhibit some relevant Cayley cycles  that arise in our context.
(References are \cite{H-L}, and \cite{J}; but we follow the
conventions of \cite{T}.)

\begin{definition} Let $M$ be a Riemannian manifold. A  closed $l$-form $\phi$ is
said to be a {\it calibration} if for every oriented tangent $l$-plane $\xi$, we have
$$
\phi|_{\xi} \le vol_{\xi}
$$
where $vol_{\xi}$ is the (Riemannian) volume form. Given a
calibration $\phi$, an oriented submanifold $N$ is said to be {\it
calibrated} if $\phi$ restricts to $N$ as the Riemannian volume
form.
\end{definition}

It is easy to see that a calibrated submanifold is minimal. Two
examples are relevant. First, if  $M$ is K\"ahler, with K\"ahler
form $\omega$, for any integer $p \ge 1$, the form
$\frac{\omega^p}{p!}$ is a calibration, and the calibrated
submanifolds are precisely the complex submanifolds.

The case that concerns us is that of a four-fold $X$ with trivial
canonical bundle $K_X$. We fix an integral Ricci-flat K\"ahler form
$\omega$, and let $\theta$ denote a  trivialisation of $K_X$  with normalisation
as in (\ref{omegatheta}). Then $4 Re(\theta)$ is a second calibration, and the
calibrated submanifolds are called {\it Special Lagrangian
submanifolds.} There is a ``linear combination" of the two, defined
by the form
$$
\Omega=\frac{w^2}{2}+4Re(\theta)
$$
which defines the {\it Cayley calibration}. The corresponding
calibrated manifolds are called {\it Cayley manifolds}. Any smooth
complex surface (on which the second term will restrict to zero) or
any Special Lagrangian submanifold (on which the first term will
vanish) furnish examples. In fact, the Cayley cycles we deal with will be of the latter kind.

Cayley manifolds are not easy to find. We will use the following
result (Proposition 8.4.8 of \cite{J}):

\begin{proposition} Let $X$ be as above, and $\sigma:X \to X$
an anti-holomorphic isometric involution such that $\sigma^*\theta =
\bar{\theta}$. Then the fixed point set is a Special Lagrangian
submanifold.
\end{proposition}

We return to the constructions of our paper. Recall that the field $F$
is invariant under complex conjugation, which therefore acts on it
with fixed field $F_1$. This induces an involution $\hat{\sigma_1}:V
\to V$ such that $\mathbf{z}(\hat{\sigma_1} (u))=
\bar{\mathbf{z}}(u)$, where, if $\mathbf{z}=(z_1,z_2,z_3,z_4)$, we
set $\bar{\mathbf{z}}=(\bz_1, \bz_2,\bz_3,\bz_4)$.  The induced
involution $\sigma_1:X \to X$ has fixed locus which we will denote
$Y$. Note that $\sigma$ satisfies the conditions of the previous
Proposition and therefore $Y$ is Special Lagrangian.

\begin{theorem}
There exist (rational) Cayley cycles representing the Hodge classes $M_i$.
\end{theorem}

\begin{proof}
Recall the isogeny $\pi_a:X \to X$, given by multiplication by the algebraic integer $a=1+i\mathbb{D}$. It is easy to check
\begin{equation*}
\begin{split}
\pi_a^* \omega = (1+\Delta) \omega\\
\pi_a^* \theta = (1+\Delta)^2 \theta\\
\end{split}
\end{equation*}
We will also need a second isogeny $\pi_b$, where $b=i\mathbb{D}$, which satisfies
\begin{equation*}
\begin{split}
\pi_b^* \omega = \Delta \omega\\
\pi_b^* \theta = \Delta^2 \theta\\
\end{split}
\end{equation*}
These equations guarantee the maps $\pi_a, \pi_b$ take Cayley cycles to Cayley cycles (possibly introducing singularities.)

We have the following table giving the action of the above isogenies on four-forms of various types (all the forms in the list are eigenvectors):
\begin{center}
\begin{tabular}{ | l | l | l | l |}
\hline
Form & eigenvalue of $\pi_a^*$ & eigenvalue of $\pi_b^*$  & ``multiplicity"\\ \hline
$dz_1 dz_2 dz_3 dz_4$  & $(1+\Delta)^2$ & $\Delta^2$ & $2 \times 1$ \\ \hline
$dz_1 d\bz_1 dz_2 dz_3$  & $(1+\Delta)^2$ & $\Delta^2$ & $2 \times 8$ \\ \hline
$dz_1 d\bz_1 dz_2 dz_4$  & $(1+\Delta)(1-i\mathbb{D})^2$ & $-\Delta^2$ & $2 \times 4$ \\ \hline
$d\bz_1 dz_2 dz_3 dz_4$  & $(1+\Delta)(1-i\mathbb{D})^2$ & $-\Delta^2$ & $2 \times 4$  \\ \hline
$dz_1 d\bz_1 dz_2 d\bz_2$  & $(1+\Delta)^2$ & $\Delta^2$ & 6 \\ \hline
$dz_1 d\bz_1 dz_2 d\bz_3$  & $(1+\Delta)(1-i\mathbb{D})^2$ & $-\Delta^2$ & $2 \times 12$ \\ \hline
$dz_1 dz_2 d\bz_3 d\bz_4$  & $(1+\Delta)^2$ & $\Delta^2$ & 4 \\ \hline
$dz_1 d\bz_2 dz_3 d\bz_4$  & $(1+i\mathbb{D})^4$ & $\Delta^2$  & 1\\ \hline
$d\bz_1 dz_2 d\bz_3 dz_4$  & $(1-i\mathbb{D})^4$ & $\Delta^2$  & 1\\ \hline
\end{tabular}
\end{center}
(We list only forms of type (4,0), (3,1) and (2,2), omitting types that are related to the ones in the list by conjugation. The term ``multiplicity'' refers to the number of forms of a given type, not the multiplicity of eigenvalues.)

Consider the operator
$$
\Phi_a=(\pi_a^*-(1+\Delta)^2)(\pi_b^*+\Delta^2)
$$
From the list it follows that the space $\cM \otimes_{\mathbb{Q}} \mathbb{C}$ (spanned by the $M_i$)
is the sum of the eigenspaces of $\Phi_a$ corresponding to the non-zero eigenvalues. We have
(using (\ref{isogeny}))
\begin{equation*}
\begin{split}
\Phi_a^* M' &= -8 \Delta^2[2 \Delta M'+ (1-\Delta) M]\\
\Phi_a^* M  &= - 8 \Delta^2[-(1-\Delta)\Delta M'+2 \Delta M]\\
\end{split}
\end{equation*}

Next, note that the Cayley cycle $Y$ defined above satisfies
\begin{equation*}
\begin{split}
<Y, M> &= 2\mathbb{D} \delta\\
<Y, M'>  &= 0 \\
\end{split}
\end{equation*}
Here $<,>$ denotes the integration pairing of cycles and forms, and $\delta$ denotes the co-volume of the lattice $\mathfrak{o}_{F_1} \subset F_1 \otimes_{\mathbb{Q}} \mathbb{R} $. By standard facts in algebraic number theory, $\delta$ is a rational multiple of $\mathbb{D}$; so the above pairings are rational, as they had better be.

We now consider the Cayley cycle
$$
C_a= (\pi_a-(1+\Delta)^2)(\pi_b+\Delta^2) Y
$$
By construction $C_a$ is orthogonal to all the forms in the above list except the $M_i$. Its pairings with these are as follows:
\begin{equation*}
\begin{split}
<C_a, M> &= -32 \Delta^3 \mathbb{D} \delta\\
<C_a, M'> & = -16 \Delta^2(1-\Delta)\mathbb{D} \delta \\
\end{split}
\end{equation*}

Let now $\ba=(1-i\mathbb{D})$, and repeat the above construction with operators $\Phi_{\ba}$, etc. \begin{equation*}
\begin{split}
\Phi_{\ba}^* M' &= -8 \Delta^2[2 \Delta M'-(1-\Delta) M]\\
\Phi_{\ba}^* M &= - 8 \Delta^2[(1-\Delta)\Delta M'+2 \Delta M]\\
\end{split}
\end{equation*}
This gives a cycle $C_{\ba}$ satisfying
\begin{equation*}
\begin{split}
<C_{\ba}, M> &= -32 \Delta^3 \mathbb{D} \delta\\
<C_{\ba}, M'> & = 16 \Delta^2(1-\Delta)\mathbb{D} \delta \\
\end{split}
\end{equation*}

Clearly the theorem is proved.\end{proof}

\begin{remark} The above result, though suggestive, does not take us
far. This is because the above ``Cayley cycle" is not effective, but
in fact a linear combination of SL subvarieties with both positive
and negative coefficients. (D. Joyce has pointed out that this must
be the case given that it represents a $(2,2)$ class.) To make
matters worse, a theorem of G. Tian (Theorem 4.3.3 of \cite{T})
states that blow-up loci of Hermite-Yang-Mills connections are
effective \emph{holomorphic} integral cycles consisting of complex
subvarieties of codimension two. So any glueing will call for very
new techniques.
\end{remark}

\section{Adapted complex structures}

In this section we seek translation-invariant complex structures on
the eight-torus $V/\Lambda$ such that the classes $A_i$ are of type
$(1,1)$ w.r.to these complex structures, and therefore define
holomorphic structures on the line bundles $\cL_i$. The original
motivation was to exploit twistor techniques for the construction of
instantons, but we postpone discussion of possible uses of this
investigation to the last section.

Consider a linear change of coordinates of the form
\begin{equation*}
\begin{split}
z_1&=w_1 + \overline{\alpha_{\bar{1}2}}\bar{w}_2+\overline{\alpha_{\bar{1}4}}\bar{w}_4\\
z_3&=w_3 + \overline{\alpha_{\bar{3}2}}\bar{w}_2+\overline{\alpha_{\bar{3}4}}\bar{w}_4\\
z_2&=w_2 + \overline{\tilde{\alpha}_{\bar{2}1}}\bar{w}_1+\overline{\tilde{\alpha}_{\bar{2}3}}\bar{w}_3\\
z_4&=w_4 + \overline{\tilde{\alpha}_{\bar{4}1}}\bar{w}_1+\overline{\tilde{\alpha}_{\bar{4}3}}\bar{w}_3\\
\end{split}
\end{equation*}
We collect the coefficients into $2 \times 2$ matrices $\alpha$ and
$\tilde{\alpha}$ as follows:
$$
 \alpha=  \left( \begin{matrix}
       \alpha_{\bar{1}2} & \alpha_{\bar{1}4}\\
       \alpha_{\bar{3}2} & \alpha_{\bar{3}4} \\
    \end{matrix} \right)
$$
and
$$
 \tilde\alpha=  \left(  \begin{matrix}
       \tilde\alpha_{\bar{2}1} & \tilde\alpha_{\bar{2}3}\\
       \tilde\alpha_{\bar{4}1} & \tilde\alpha_{\bar{4}3} \\
    \end{matrix} \right)
 $$
 and rewrite the above change of coordinates as follows:
\begin{equation*}
\begin{split}
 \left ( \begin{matrix}
      z_1  \\
      z_3 \\
   \end{matrix} \right)
   &=
\left(  \begin{matrix}
      w_1 \\
      w_3 \\
   \end{matrix} \right)
   +
 \bar\alpha
\left( \begin{matrix}
     \bar{w}_2 \\
      \bar{w}_4 \\
   \end{matrix} \right)\\
\left( \begin{matrix}
      z_2  \\
      z_4 \\
   \end{matrix} \right)
   &=
\left( \begin{matrix}
      w_2 \\
      w_4 \\
   \end{matrix} \right)
   +
 \bar{\tilde\alpha}
\left( \begin{matrix}
     \bar{w}_1 \\
      \bar{w}_3 \\
   \end{matrix} \right)
\end{split}
\end{equation*}

A long but straightforward computation shows $A_i$ will be of type
$(1,1)$ provided:
\begin{equation*}
\begin{split}
h_3(x_1-x_3)(\alpha_{\bar{1}2}\alpha_{\bar{3}4}-\alpha_{\bar{1}4}\alpha_{\bar{3}2})&\\
+h_4(x_1-x_4)\alpha_{\bar{1}2} - h_2(x_1-x_2)\alpha_{\bar{1}4}&\\
+h_2(x_3-x_4)\alpha_{\bar{3}2} - h_4(x_3-x_2)\alpha_{\bar{3}4}&\\
+h_3(x_2-x_4)&=0
\end{split}
\end{equation*}
and
\begin{equation*}
\begin{split}
h_3(x_2-x_4)(\tilde\alpha_{\bar{2}1}\tilde\alpha_{\bar{4}3}
-\tilde\alpha_{\bar{2}3}\tilde\alpha_{\bar{4}1})&\\
+h_4(x_1-x_4)\tilde\alpha_{\bar{4}3} - h_2(x_3-x_4)\tilde\alpha_{\bar{4}1}&\\
+h_2(x_1-x_2)\tilde\alpha_{\bar{2}3} - h_4(x_3-x_2)\tilde\alpha_{\bar{2}1}&\\
+h_3(x_1-x_3)&=0
\end{split}
\end{equation*}

To rewrite these conditions in a more compact form, we introduce
some notation:
\begin{enumerate}
\item Given a $2 \times 2$ matrix $A$:
$$
A=\left( \begin{matrix}
      a_{11} & a_{12}\\
      a_{21} & a_{22}\\
   \end{matrix} \right)
$$
let
\begin{equation}
\label{hat}
\hat A=\left(  \begin{matrix}
      a_{22} & -a_{12}\\
      -a_{21} & a_{11}\\
   \end{matrix} \right)
\end{equation}
(If $A$ is nonsingular,  $\hat{A}=(det\ A)A^{-1}$.)

\item Define the
symmetric bilinear form $<,>$ on the space of $2 \times 2$ matrices
$$
<A,B>=Tr(A\hat{B})=det\ (A+B)-det\ A -det\ B
$$

\item Let
$$
H=\left(   \begin{matrix}
      -h_4(x_2-x_3) & h_2(x_3-x_4) \\
     - h_2(x_1-x_2) & -h_4(x_1-x_4) \\
   \end{matrix} \right)
$$
so that
$$
\hat H=\left(  \begin{matrix}
      -h_4(x_1-x_4) & -h_2(x_3-x_4) \\
      h_2(x_1-x_2) & -h_4(x_2-x_3) \\
   \end{matrix} \right)
$$
\end{enumerate}

The conditions on $\alpha$ and $\tilde\alpha$ can now be rewritten:
\begin{equation}
\label{condition1} <\alpha, H>=h_3(x_2-x_4)+h_3(x_1-x_3)det\ \alpha
\end{equation}
and
\begin{equation}
\label{condition2}
 <\tilde\alpha, \hat H>=h_3(x_1-x_3)+h_3(x_2-x_4)det\ \tilde\alpha
\end{equation}

We assume that the inverse coordinate transformation is of the form
\begin{equation*}
\begin{split}
 \left ( \begin{matrix}
      w_1  \\
      w_3 \\
   \end{matrix} \right)
   &=
c\left(  \begin{matrix}
      z_1 \\
      z_3 \\
   \end{matrix} \right)
   +
 \bar\beta
\left( \begin{matrix}
     \bar{z}_2 \\
      \bar{z}_4 \\
   \end{matrix} \right)\\
\left( \begin{matrix}
      w_2  \\
      w_4 \\
   \end{matrix} \right)
   &=
\tilde{c} \left( \begin{matrix}
      z_2 \\
      z_4 \\
   \end{matrix} \right)
   +
 \bar{\tilde\beta}
\left( \begin{matrix}
     \bar{z}_1 \\
      \bar{z}_3 \\
   \end{matrix} \right)
\end{split}
\end{equation*}
where $c,\ \tilde c$ are scalars (this will constrain $\alpha$ and
$\tilde \alpha$, see below) and $\beta$ and $\tilde \beta$ $2 \times
2$ matrices. One checks that we then need
\begin{equation*}
\begin{split}
c(1-\bar\alpha \tilde\alpha)&=1\\
\tilde c (1-\bar{\tilde \alpha}\alpha)&=1
\end{split}
\end{equation*}
so that we  are requiring that $\bar\alpha \tilde\alpha$ and
$\bar{\tilde \alpha}\alpha$ are scalars. Further,
\begin{equation*}
\begin{split}
\beta &=-{\tilde c} \alpha\\
\tilde \beta &= -c \tilde \alpha
\end{split}
\end{equation*}
Note that either $\bar\alpha \tilde\alpha=\bar{\tilde
\alpha}\alpha=0$ and $c=\tilde{c}=1$ or
\begin{equation*}
\begin{split}
\bar\alpha \tilde\alpha&=\frac{c-1}{c}\\
\bar{\tilde \alpha}\alpha&=\frac{\tilde{c}-1}{\tilde c}
\end{split}
\end{equation*}
and $\tilde c= \bar{c}$. Note also that once $\alpha$ is chosen to
satisfy the equation (\ref{condition1}), then (\ref{condition2}) is
satisfied if we take
$$
\tilde{\alpha}=\frac{x_1-x_3}{x_2-x_4}\hat{\bar{\alpha}}
$$
From now on we will proceed to define $\tilde{\alpha}$ by the above
equation. This forces $c$ to satisfy
$$
c(1-\frac{x_1-x_3}{x_2-x_4} \overline{det\ \alpha})=1
$$
Clearly, a necessary condition is
\begin{equation}
\label{detalpha}
det\  \alpha \ne \frac{x_2-x_4}{x_1-x_3}
\end{equation}

We can write down the corresponding almost complex structure. With
an obvious schematic notation,
\begin{equation*}
\begin{split}
 J \left ( \begin{matrix}
      dz_1  \\
      dz_3 \\
   \end{matrix} \right)
   &=
i(2c-1)\left(  \begin{matrix}
      dz_1 \\
      dz_3 \\
   \end{matrix} \right)
   +
2i \bar\beta \left( \begin{matrix}
     d\bar{z}_2 \\
      d\bar{z}_4 \\
   \end{matrix} \right)\\
J \left( \begin{matrix}
      dz_2  \\
      dz_4 \\
   \end{matrix} \right)
   &=
i(2\tilde{c}-1) \left( \begin{matrix}
      dz_2 \\
      dz_4 \\
   \end{matrix} \right)
   +
2i \bar{\tilde\beta} \left( \begin{matrix}
     d\bar{z}_1 \\
      d\bar{z}_3 \\
   \end{matrix} \right)
\end{split}
\end{equation*}

By further restricting $\alpha$ one can ensure that $\omega$ remains
of type $(1,1)$. We summarise our results in
\begin{theorem}
Let the co-ordinates $w$ be defined by
\begin{equation*}
\begin{split}
 \left ( \begin{matrix}
      z_1  \\
      z_3 \\
   \end{matrix} \right)
   &=
\left(  \begin{matrix}
      w_1 \\
      w_3 \\
   \end{matrix} \right)
   +
 \bar\alpha
\left( \begin{matrix}
     \bar{w}_2 \\
      \bar{w}_4 \\
   \end{matrix} \right)\\
\left( \begin{matrix}
      z_2  \\
      z_4 \\
   \end{matrix} \right)
   &=
\left( \begin{matrix}
      w_2 \\
      w_4 \\
   \end{matrix} \right)
   +
 \bar{\tilde\alpha}
\left( \begin{matrix}
     \bar{w}_1 \\
      \bar{w}_3 \\
   \end{matrix} \right)
\end{split}
\end{equation*}
where the matrix $\alpha$ satisfies
\begin{equation}
\label{aoneone} <\alpha, H>=h_3(x_2-x_4)+h_3(x_1-x_3)det\ \alpha
\end{equation}
and
$$
\tilde{\alpha}=\frac{x_1-x_3}{x_2-x_4}\hat{\bar{\alpha}}
$$
($\hat{\alpha}$ is defined as in (\ref{hat}).) Then the forms $A_i$
are of type $(1,1)$ w.r.to the $w_i$. Further, if $\alpha$ satisfies
\begin{equation}
\label{omegaoneone}
\begin{split}
\alpha_{\bar{3}4}&=+\frac{x_1-x_4}{x_2-x_3}\bar{\alpha}_{\bar{1}2}\\
\alpha_{\bar{1}4}&=-\frac{x_3-x_4}{x_1-x_3}\bar{\alpha}_{\bar{3}2}\\
\end{split}
\end{equation}
then $\omega$ remains of type $(1,1)$.
\end{theorem}

If $\alpha$ satisfies (\ref{omegaoneone}), the condition
(\ref{aoneone}) becomes
\begin{equation}
\begin{split}
h_3(x_1-x_3))(\frac{x_1-x_4}{x_2-x_3}|\alpha_{\bar{1}2}|^2 +
\frac{x_3-x_4}{x_1-x_2}|\alpha_{\bar{3}2}|^2)+h_3(x_2-x_4)&\\+h_4(x_1-x_4)(\alpha_{\bar{1}2}+\bar{\alpha}_{\bar{1}2})
+h_2(x_3-x_4)(\alpha_{\bar{3}2}+\bar{\alpha}_{\bar{3}2})&=0
\end{split}
\end{equation}
The space of solutions $\tilde{\mathcal{J}}$ is clearly an
3-dimensional ellipsoid in the two-dimensional complex vector space
with co-ordinates $(\alpha_{\bar{1}2},\alpha_{\bar{3}2})$. The
condition (\ref{detalpha}) becomes:
$$
\frac{x_1-x_4}{x_2-x_3}|\alpha_{\bar{1}2}|^2 +
\frac{x_3-x_4}{x_1-x_2}|\alpha_{\bar{3}2}|^2 \ne 0
$$
which corresponds to removing the affine hyperplane $\mathcal H$
given by
$$
h_3(x_2-x_4)+h_4(x_1-x_4)(\alpha_{\bar{1}2}+\bar{\alpha}_{\bar{1}2})
+h_2(x_3-x_4)(\alpha_{\bar{3}2}+\bar{\alpha}_{\bar{3}2})=0
$$
We have therefore to consider $\mathcal J = \tilde{\mathcal J}
\setminus \mathcal H$, which is the union of two open three-discs.

A particular choice of $\alpha$ has remarkable properties. Let
$$
\alpha^*=\frac{1}{x_1-x_3}\left(\begin{matrix}
      (x_2-x_3)(1-\frac{2h_4}{h_3}) & -(x_3-x_4)(1-\frac{2h_2}{h_3})\\
      (x_1-x_2)(1-\frac{2h_2}{h_3}) & (x_1-x_4)(1-\frac{2h_4}{h_3}) \\
   \end{matrix} \right)
$$
\begin{theorem}
With this choice, we have
\begin{equation}
\begin{split}
 \frac{(x_1  x_3 + x_2  x_4)^2}{4}A_1=&
 {(x_1 - x_2)^2(x_2 - x_3)^2(x_1 - x_4)^2(x_3 - x_4)}
dw_1  d\bw_2\\
-& {(x_1 - x_2)^2(x_2 - x_3)^2(x_1 - x_4)(x_3 - x_4)^2}
 dw_2 d\bw_3 \\
 +& {(x_1 - x_2)(x_2 - x_3)^2(x_1 - x_4)^2(x_3 - x_4)^2}
dw_3  d\bw_4\\
+& {(x_1 - x_2)^2(x_2 - x_3)(x_1 - x_4)^2(x_3 - x_4)^2}
 dw_4  d\bw_1\\
+& {(x_1 - x_2)^2(x_2 - x_3)^2(x_1 - x_4)^2(x_3 - x_4)}
d\bw_1 dw_2  \\
-& {(x_1 - x_2)^2(x_2 - x_3)^2(x_1 - x_4)(x_3 - x_4)^2}
d\bw_2 dw_3  \\
+& {(x_1 - x_2)(x_2 - x_3)^2(x_1 - x_4)^2(x_3 - x_4)^2}
d\bw_3 dw_4  \\
+& {(x_1 - x_2)^2(x_2 - x_3)(x_1 - x_4)^2(x_3 - x_4)^2}
 d\bw_4 dw_1 \\
\end{split}
\end{equation}

\begin{equation}
\begin{split}
\frac{(x_1  x_3 + x_2  x_4)^2}{4i\mathbb{D}}A_2=&
 {(x_1 - x_2)^2(x_2 - x_3)^2(x_1 - x_4)^2(x_3 - x_4)}
dw_1  d\bw_2\\
+& {(x_1 - x_2)^2(x_2 - x_3)^2(x_1 - x_4)(x_3 - x_4)^2}
 dw_2 d\bw_3 \\
 +& {(x_1 - x_2)(x_2 - x_3)^2(x_1 - x_4)^2(x_3 - x_4)^2}
dw_3  d\bw_4\\
-& {(x_1 - x_2)^2(x_2 - x_3)(x_1 - x_4)^2(x_3 - x_4)^2}
 dw_4  d\bw_1\\
-& {(x_1 - x_2)^2(x_2 - x_3)^2(x_1 - x_4)^2(x_3 - x_4)}
d\bw_1 dw_2  \\
-& {(x_1 - x_2)^2(x_2 - x_3)^2(x_1 - x_4)(x_3 - x_4)^2}
d\bw_2 dw_3  \\
-& {(x_1 - x_2)(x_2 - x_3)^2(x_1 - x_4)^2(x_3 - x_4)^2}
d\bw_3 dw_4  \\
+& {(x_1 - x_2)^2(x_2 - x_3)(x_1 - x_4)^2(x_3 - x_4)^2}
 d\bw_4 dw_1 \\
\end{split}
\end{equation}

\begin{equation}
\begin{split}
\frac{\Delta(x_1  x_3 + x_2  x_4)^2}{4i\mathbb{D}}\omega\\=-\{ &
{(x_1 - x_2)^2(x_1 - x_3)(x_2 - x_3)(x_1 - x_4)^2(x_3 - x_4)}
dw_1  d\bw_1\\
+
& {(x_1 - x_2)^2(x_2 - x_3)^2(x_1 - x_4)(x_2 - x_4)(x_3 - x_4)}
dw_2 d\bw_2 \\
+
& {(x_1 - x_2)(x_1 - x_3)(x_2 - x_3)^2(x_1 - x_4)(x_3 - x_4)^2}
dw_3  d\bw_3 \\
+
& {(x_1 - x_2)(x_2 - x_3)(x_1 - x_4)^2(x_2 - x_4)(x_3-x_4)^2}
dw_4  d\bw_4\} \\
\end{split}
\end{equation}
In particular, $-\omega$ is a K\"ahler form and the corresponding
complex structure makes $X_r$ an abelian variety.
\end{theorem}

\begin{remark}\label{flip} It is convenient to consider the consider
the \emph{conjugate} complex structure (w.r.to which holomorphic
co-ordinates are the $\bw_i$. This has the property that the forms
$A_i$ and $\omega$ are of type (1,1), and \emph{in addition,
$\omega$ is K\"ahler.} We let $X'$ denote the corresponding abelian
variety.
\end{remark}

\section{A strategy}

Attempts to invoke twistor methods have not been successful so far.
For example, N. Hitchin pointed out that results of M. Verbitsky
make hyperk\"ahler twistor spaces quite unsuitable. G. Tian made the
following suggestion: construct instantons by deformation (using,
say, the continuity method) from a situation when they are known to
exist. In fact, the complex structure described in the Remark
\ref{flip} provides such a starting point. I close with a brief
justification for this claim.

With respect to the above complex structure, the bundles $\hcV_i$
defined in the statement of Theorem \ref{bundles} are holomorphic,
and furthermore (using the ampleness of $\omega$), the constant $k$
can be chosen large enough that $\hcV_2$ can be embedded as a
sub-bundle of ${\hcV_1}^{c^2\Delta}$. The quotient bundle can be
identified with the difference bundle $\cE$, which therefore has a
holomorphic structure depending on the above embedding; \emph{we now
show that it is possible to arrange that $\cE$, endowed with this
structure, is polystable.} (By stability we shall mean
$\mu$-stability w.r.to the polarisation $\omega$.

Let us start by recalling that $\hcV_1=L_1 \otimes L_{k\omega}
\oplus L_1^{-1} \otimes L_{k\omega} $, and $\hcV_2= L_2 \oplus
L_2^{-1}$. Choose a large enough integer $k_1$ such that
$L_{k_1\omega}$ that is very ample, and let $C$ be a general curve
cut out by three sections of this line bundle. It follows from
Proposition \ref{miracle2} that $d \equiv degree\ L_2^{-1}\otimes
L_1 \otimes L_{k\omega}|_C=degree\ L_2\otimes L_1^{-1} \otimes
L_{k\omega}|_C=degree\ L_{k\omega}|_C=kk_1^3<\omega^4>$, and will
assume that $k$ is chosen such that $d >
2genus(C)=3k_1^4<\omega^4>+2$. We next make the following
assumption:
\begin{equation}\label{linsys}
dim\ H^0(C,L_{k\omega}|_C)=c^2\Delta
\end{equation}
which we will return to below. Let $W$ denote a subspace of
$H^0(X',L_2^{-1}\otimes L_1 \otimes L_{k\omega})$, chosen such that
\begin{itemize}
  \item the restriction map $W  \to H^0(C,L_2^{-1}\otimes L_1 \otimes
  L_{k\omega}|_C)$ is an isomorphism, and
  \item $W$ is base-point free.
\end{itemize}
Consider now the evaluation map $E:W \otimes \cO_{X'} \to
L_2^{-1}\otimes L_1 \otimes L_{k\omega}$, and let $\cF$ be the
kernel; by construction $\cF$ fits in the exact sequence
$$
0 \to \cF \to W \otimes \cO_{X'} \to L_2^{-1}\otimes L_1 \otimes
L_{k\omega} \to 0\ .
$$
By Butler's Theorem (\cite{Bu}), the restriction of $\cF$ to $C$ is
stable, and this proves that $\cF$ itself is stable. We next choose
a subspace $U$ of $H^0(X',L_2\otimes L_1^{-1} \otimes L_{k\omega})$
with similar properties and obtain a second stable bundle $\cG$ that
fits in the sequence
$$
0 \to \cG \to U \otimes \cO_{X'} \to L_2\otimes L_1^{-1} \otimes
L_{k\omega} \to 0
$$
Dualising, tensoring by suitable line bundles and adding the two
sequences, we get
$$
0 \to \hcV_2 \to \hcV_1^{c^2\Delta} \to \hcF \otimes L_1 \otimes
L_{k\omega} \oplus \hcG \otimes L_1^{-1} \otimes L_{k\omega} \to 0
$$
where $\hcF$ denotes the dual of $\cF$ and $\hcG$ denotes the dual
of $\cG$, and we have used the assumption (\ref{linsys}), namely,
$dim\ W = dim\ U =c^2 \Delta$. Repeatedly using Proposition
\ref{miracle2} we see that the two summands in the last sum have the
same slope. Consider now the assumption (\ref{linsys}). By
Riemann-Roch, this is equivalent to:
$$
(kk_1^3-(3/2)k_1^4)<\omega^4>=c^2\Delta
$$
This is solved by taking
$$
k=(\frac{c^2\Delta}{k_1^3}+\frac{3k_1}{2})/<\omega^4>
$$
This is where the choice of $c$ comes in - we choose $c$ and $k_1$
such that $k$ is an integer (and large enough). Once this is done
\begin{theorem}
The bundle $\cE$ (on $X'$) can be given a holomorphic structure such that
it is polystable.
\end{theorem}

The above application of Butler's theorem is inspired by its use in
\cite{M}.

By Donaldson-Uhlenbeck-Yau, such a bundle would admit a
Hermite-Einstein metric and therefore a Spin(7) instanton.

\end{document}